\documentclass[11pt]{amsart}
\usepackage{amsthm,amsmath,amsxtra,amscd,amssymb,color,braket}
\usepackage[a4paper,top=3cm,bottom=4cm]{geometry} 
\usepackage{amsaddr} 
\newcommand{\al}{\alpha}

\newcommand{\de}{\delta}
\newcommand{\e}{\varepsilon}
\newcommand{\ep}{\varepsilon}
\newcommand{\lm}{\lambda}

\newcommand{\p}[1]{\partial^{[#1]}}
\newcommand{\HH}{{\mathbb{H}}}

\newcommand{\CC}{{\mathbb{C}}}

\newcommand{\NN}{{\mathbb{N}}}
\newcommand{\ZZ}{{\mathbb{Z}}}
\newcommand{\RR}{{\mathbb{R}}}
\newcommand{\op}{\operatorname}
\newcommand\grad{\op{grad}}
\newcommand\diag{\op{diag}}
\newcommand{\T}{\Theta}
\newcommand{\GL}{\op{GL}_g(\CC)}
\newcommand{\inv}{^{-1}}

\newcommand{\ch}[2]{\left[\begin{smallmatrix}#1\\#2\end{smallmatrix}\right]}
\newcommand\tch[1]{{\vartheta\left[\begin{smallmatrix}#1\end{smallmatrix}\right]}}
\newcommand{\sm}[2]{\left(\begin{smallmatrix}#1\\#2\end{smallmatrix}\right)}

\theoremstyle{plain}
\newtheorem{thm}{Theorem}
\newtheorem{cor}{Corollary}
\newtheorem{lem}{Lemma}
\newtheorem{prop}{Proposition}

\theoremstyle{remark}
\newtheorem*{rem}{Remark}

\title{Heat equation for theta functions and vector-valued modular forms}
\author{Sara Perna}
\address{Universit\`a degli Studi di Roma ``La Sapienza'', \\Dipartimento di Matematica ``Guido Castelnuovo'', \\Piazzale A. Moro 5, 00185, Roma, Italia.}
\email{perna@mat.uniroma1.it}

\begin{document}
\begin{abstract}
We give a new method for constructing vector-valued modular forms 
 from singular scalar-valued ones. As an application we prove the 
identity between two remarkable spaces of vector-valued modular forms which  seem to be unrelated at a first look, since they are constructed in two very different ways. If $V_{grad}$ is the vector space generated by vector-valued modular forms constructed with gradients of odd theta functions and $V_\T$ is the one generated by vector-valued modular forms arising from second order theta constants with our new construction, we will prove that $V_{grad}=V_\T$. This result could also be proven as a consequence of the ``heat equation'' for theta functions. 
\end{abstract}

\maketitle
\section{Introduction}

Some constructions of vector-valued modular forms with respect to the integral symplectic group from scalar-valued ones have been recently investigated in~\cite{CvG} and~\cite{CFvG}. 

In this paper we will construct vector-valued modular forms with respect to a congruence subgroup of the integral symplectic group from singular scalar-valued modular forms.
This new construction comes from a development of the ideas in~\cite[Section 5]{FDP}, where we provided a link between two apparently unrelated methods of constructing holomorphic differential forms on suitable modular varieties.

Let us introduce some notations. 
Let $\HH_g$ denote the Siegel space of degree $g$. This is the space of $g\times g$ complex symmetric matrices with positive definite imaginary part. The group of integral symplectic matrices $\Gamma_g:=\op{Sp}(2g,\ZZ)$ acts properly discontinuously on $\HH_g$.
The action of $\gamma=\sm{A&B}{C&D}\in\Gamma_g$, where $A,\,B,\,C,\,D$ are $g\times g$ matrices, on a point $\tau\in\HH_g$ is defined as
\begin{equation}\label{symplectic action}
\gamma\cdot\tau=(A\tau+B)(C\tau+D)^{-1}.
\end{equation}
We will keep this block notation for an element of $\Gamma_g$ throughout the paper. If $\Gamma\subset\Gamma_g$ is a group acting properly discontinuously on the Siegel space, the quotient $\HH_g/\Gamma$ is called modular variety. It has the structure of a normal analytic space and it is a quasi-projective variety. 

Modular forms are important tools for studying the geometry of these varieties. For example, vector-valued modular forms are strictly related to the definition of holomorphic differential forms on them. 

If $(\rho,V_\rho)$ is a finite dimensional rational representation of $\GL$, a vector-valued modular form with respect to $\rho$ is a holomorphic function $f:\HH_g\to V_\rho$ such that
\[f(\gamma\cdot\tau)=\rho(C\tau+D)\,f(\tau),\]
for any $\gamma\in\Gamma$ and for any $\tau\in\HH_g$
If $\rho=\det^{k/2}$ for some $k\in\NN$, $f$ is said to be a scalar-valued modular form of weight $k/2$. It will be called singular if $k<g$. The complex vector space of such modular forms will be denoted by $[\Gamma,k/2]$. 
 
If $X$ is a complex manifold, denote by $\Omega^n(X)$ the space of holomorphic differential forms on $X$ of degree $n$.  If $g\ge2$ and $n<g(g+1)/2$, there is a natural isomorphism 
\[\Omega^n(X_\Gamma^0)\cong\Omega^n(\HH_g)^\Gamma,\]
where $X_\Gamma^0$ is the set of regular points of $\HH_g/\Gamma$ and $\Omega^n(\HH_g)^\Gamma$ is the space of $\Gamma$-invariant holomorphic differential forms on $\HH_g$ of degree $n$ (cf.~\cite{FP}).
For suitable degrees some of these spaces are known to be trivial. The possible non-trivial spaces are identified with vector spaces of vector-valued modular forms (cf.~\cite{weissauer}). 
For $N=g(g+1)/2$ the identification of $\Gamma$-invariant holomorphic differential forms of degree $N-1$ is given in the following way. 
Let $\Set{d\check{\tau}_{ij}}_{i,j=1}^g$ be the basis of holomorphic differential forms on $\HH_g$ of degree $N-1$ given by
\begin{equation}\label{basis}
d\check{\tau}_{ij}=\pm\; e_{ij}\bigwedge_{\substack{1\le k\le l\le g\\[2pt](k,l)\neq(i,j)}}d\tau_{kl};\;\; e_{ij}=\frac{1+\de_{ij}}{2},
\end{equation}
where the sign is chosen in such a way that  $d\check{\tau}_{ij}\wedge d\tau_{ij}=e_{ij}\bigwedge_{1\le k\le l\le g}d\tau_{kl}$. 
By~\cite{F83} a differential form $\omega\in\Omega^{N-1}(\HH_g)$ is $\Gamma$-invariant if and only if 
\[\omega=\op{Tr}(A(\tau)d\check{\tau}),\]
where $A(\tau)$ is a vector-valued modular form 
 satisfying the transformation rule
\begin{equation}\label{trans_diff}
A(\gamma\cdot\tau)=\op{det}(C\tau+D)^{g+1}\,{}^t(C\tau+D)^{-1}\,A(\tau)\,(C\tau+D)^{-1},
\end{equation}
for any $\gamma\in\Gamma$ and $\tau\in\HH_g$.

Scalar-valued modular forms of suitable weight can be used to define holomorphic differential forms of degree $N-1$ invariant under the action of a group $\Gamma\subset\Gamma_g$, as explained in~\cite{F75}.
For any $f,\,h\in[\Gamma,(g-1)/2]$, the author defines a holomorphic differential form  in $\Omega^{N-1}(\HH_g)^\Gamma$ by applying suitable differential operators to the two scalar-valued modular forms. We will denote by $\omega_{f,\,h}$ such holomorphic differential form. Note that $f$ and $h$ are singular scalar-valued modular forms. This method produces also $\Gamma_g$-invariant holomorphic differential forms for $g=8k+1$, $k\ge 1$ (cf.~\cite{F75,FDP}).  

A second method of constructing elements of $\Omega^{N-1}(\HH_g)^{\Gamma_g}$ is examined in~\cite{SM}. The author starts from gradients of odd theta functions and produces holomorphic differential forms invariant under the action of the full modular group for $g\equiv 0\pmod 4$, $g\neq 5,\,13$.  

These two methods seemed to be totally unrelated until~\cite{FDP} provided a link between them. The key point in the proof is that theta functions satisfy the heat equation.

In this paper we focus on vector-valued modular forms and not only on invariant holomorphic differential forms. We give a new method for constructing vector-valued modular forms from singular scalar-valued ones. As an application, we will see that the result obtained in~\cite[Theorem 14]{FDP} is indeed a particular instance of a more general result. 
%
Let us briefly outline the results of the paper. 
For $f,\,h\in[\Gamma,1/2]$ define
\[A_{f,\,h}=
f(\partial h)-(\partial f)h,\]
where $\partial:=(\partial_{ij})$ is the $g\times g$ matrix of differential operators
\[\partial_{ij}=\begin{cases}
\frac{\partial}{\partial\tau_{ij}}& i=j\\
\frac{1}{2}\,\frac{\partial}{\partial\tau_{ij}}& i\neq j
\end{cases}.
\]
It is easy to prove that $A_{f,\,h}$ is a vector-valued modular form that satisfies the transformation rule 
\[A_{f,\,h}(\gamma\cdot\tau)=\det(C\tau+D)\,(C\tau+D)\,A_{f,\,h}(\tau)\,{}^t(C\tau+D),\]
for all $\gamma\in\Gamma$ and $\tau\in\HH_g$.
If $f_i,\,h_i$, with $1\le i\le k<g$, are in $[\Gamma,1/2]$, 
we will define a product $\ast$ such that
\begin{equation*}
A_{f_1,\,h_1}\ast\dots\ast A_{f_k,\,h_k}
\end{equation*}
is a vector-valued modular form with respect to a suitable irreducible representation $\rho_k$. If $k=g-1$, the representation $\rho_k$ is the one appearing in~\eqref{trans_diff}.

Generalizing the method in~\cite{F75}, for $f$ and $h$ scalar-valued modular forms of weight $k/2$, with $1\le k< g$, we will define  two pairings that we will denote by $\{f,h\}_k$ and by $[f,h]_k$. 
For $k=g-1$ these are the pairings appearing in~\cite{F75} for the construction of holomorphic differential forms. Indeed, 
for $f,\,h\in[\Gamma,(g-1)/2]$, one has that
\begin{equation}
\omega_{f,h}=\{f,h\}_{g-1}\sqcap d\check{\tau}=\op{Tr}([f,h]_{g-1}\,d\check{\tau}),
\end{equation}
where $d\check{\tau}$ is the basis of $\Omega^{N-1}(\HH_g)$ given in~\eqref{basis} and $\sqcap$ is a suitably defined product.

If $f=\prod_{i=1}^kf_i$ and $h=\prod_{i=1}^kh_i$ with $f_i$ and $h_i$ of weight $1/2$ then we will prove that
\[[f,h]_k=\sum_{\sigma\in S_k} A_{f_1,\,h_{\sigma(1)}}\ast\dots\ast A_{f_k,\,h_{\sigma(k)}},\]
where $S_k$ is the group of permutations of the set $\Set{1,\dots,k}$. 
The proof involves some computations of suitable differential operators applied to scalar-valued modular forms and some results about the rank of singular scalar-valued  modular forms. 

We will apply this constructions to a remarkable type of scalar-valued modular forms: second order theta constants. Theta functions and theta constants gives important examples of scalar-valued and vector-valued modular forms. 
Concerning theta functions, vector-valued modular forms constructed from gradients of odd theta functions are presented in~\cite{sm89}, generalizing the method in~\cite{SM}. 
We will prove that the relationship between the two methods in~\cite{F75} and~\cite{SM} given in~\cite{FDP} is not only at the level of holomorphic differential forms but also at the level of vector-valued modular forms. More precisely, denote by $V_{grad}$ the vector space generated by the vector-valued modular forms constructed with gradients of odd theta functions and by $V_\T$ the vector space generated by the vector-valued modular forms constructed with our new method applied to second order theta constants. We will prove that $V_{grad}=V_\T$. In order to prove this result we will make use of a generalization of Jacobi's derivative formula given in~\cite{GSM}.

\section*{Acknowledgements}

The author is grateful to Professor E. Freitag for reading a first version of the manuscript. Also, special thanks are due to F. Dalla Piazza,  A. Fiorentino, S. Grushevsky for stimulating discussions. 

\section{Multilinear algebra}\label{multilinear}
In this section we present some results in multilinear algebra. First we shall fix some notations. 

For $X\subset \NN$ of finite cardinality, denote by $P^*_k(X)$ the collection of the increasingly ordered subsets of $X$ with fixed cardinality $k$. If $I\in P^*_k(X)$ set $I^c:=X\setminus I\in P^*_{n-k}(X)$, where $n$ is the cardinality of $X$. Denote by $X_g$ the ordered set $\{1,\dots,g\}$. 

If $M$ is a $g\times g$ matrix its elements will be denoted by $M^i_j$ where $i$ is the row index and $j$ is the column index. If $I\in P^*_k(X_g)$ and $J\in P^*_l(X_g)$ denote by $M(I,J)$ the $k\times l$ submatrix of M obtained by taking rows in $I$ and columns in $J$. If $J=\{j_i,\dots,j_l\}$ we will write
\[M(I,J)=(M(I,{j_1})\mid\dots\mid M(I,{j_l}))\]
to emphasize the columns of the submatrix. If $I=X_g$ we will write $M_J$ for $M(I,J)$. 

The following formula is a well known generalization of the Laplace expansion theorem for the determinant of a square matrix.
Choose $1\le k< g$ and fix $J\in P^*_k(X_g)$  then
\begin{equation}\label{expansion}
\op{det}(M)=\sum_{I\in P^*_k(X_g)}(-1)^{I+J}
\det(\, M(I,J)\,)\,\det(\, M(I^c,J^c)\,),
\end{equation}
where $I+J$ means the sum of all the indexes in $I$ and $J$.
Here we are fixing a set of columns of $M$ and extracting minors of order $k$ from such columns with the related cofactors, the same formula holds if we fix a set of rows and extract from them minors of order $k$.

Denote by $M^{(k)}$ the matrix of cofactors of submatrices of order $k<g$ of $M$. We will index the entries of $M^{(k)}$ by some sets of indexes, that is we will write
\[M^{(k)}=\Set{(M^{(k)})^I_J}_{I,J\in P^\ast_k(X_g)},\]
 where
\begin{equation*}
(\,M^{(k)}\,)^I_J=(-1)^{I+J}\det(\, M(I^c,J^c)\,),
\end{equation*}
for $I,J\in P^*_k(X_g)$. 
This notation is justified by the relation with exterior powers of linear mapping, relation that we will explain in the following.
For $k=0$ we set $M^{(0)}=\det{M}$. Moreover, ${}^tM^{(1)}$ is the adjoint matrix of $M$, that is the matrix such that 
\begin{equation*}
M\,{}^t\big(M^{(1)}\,\big)=(\det M)\,\mathbf{1}_g,
\end{equation*}
where $\mathbf{1}_g$ is the identity matrix of size $g$.

Let $V$ be a $g$-dimensional complex vector space and fix a basis $\Set{e_i}_{i=1}^g$. If $L:V\to V$ is a linear map, then for any $1\le p\le g$ there is an associated linear map $\bigwedge^pL:\bigwedge^pV\to\bigwedge^pV$. If the map $L$ is given by a matrix $M$ with respect to the fixed basis of $V$, the matrix of the associated map $\bigwedge^pL$ with respect to the basis
\begin{equation}\label{basis_wedge}
e_I=e_{i_1}\wedge\dots\wedge e_{i_p},\;I=\{i_1,\dots,i_p\}\in P^*_p(X_g).
\end{equation}
will be denoted by $\bigwedge^p M$. It can be easily obtained by the matrix $M$. Indeed, the entries of $\bigwedge^pM$ are given by
\begin{equation}\label{wedge}
(\textstyle{\bigwedge^p}M)^I_J=\op{det}(M(I,J)),\;I,J\in P^*_p(X_g)
\end{equation}
Note that when we work with exterior powers of vector spaces, the elements of a matrix representing a linear map are indexed  by some set of indexes corresponding to the indexing set of the chosen basis~\eqref{basis_wedge}. 

We will now introduce the fundamental tool in all the computations of the paper. 
If $A:\bigwedge^pV\to \wedge^pV$ and $B:\wedge^qV\to \wedge^qV$ one can 
define the linear map 
\[A\sqcap B:\wedge^{p+q}V\to \wedge^{p+q}V\]
given by the following matrix (cf.~\cite{F75})
\begin{equation}\label{prod}
(A\sqcap B)^H_K=\frac{1}{\binom{p+q}{p}}\sum_{\substack{I\in P^*_p(H)\\J\in P^*_p(K)}}(-1)^{I+J}A^{I}_{J}\;B^{I^c}_{J^c}
,\quad H,\,K\in P^*_{p+q}(X_g).
\end{equation}

\begin{lem}\label{lemma}
For $1\le k\le g$, fix $I\in P^*_k(X_g)$ and $J=\{j_1,\dots,j_k\}\in P^*_k(X_g)$. If $A_1,\dots,A_k:V\to V$ then
\begin{equation*}
(A_1\sqcap\dots\sqcap A_k)^I_J=\frac{1}{k!}\sum_{\sigma\in S_k}\epsilon(\sigma)\op{det}(\mathbf{A}_\sigma),
\end{equation*}
where $\epsilon(\sigma)$ is the sign of the permutation $\sigma$ and 
\[\mathbf{A}_\sigma=\left(A_1(I,j_{\sigma(1)})\mid\dots\mid A_k(I,j_{\sigma(k)})\right).\]
\end{lem}

\proof
We proceed by induction on $k$. The case $k=2$ follows directly from the definition~\eqref{prod}. For $k\ge3$, a direct computation from~\eqref{prod} and the inductive argument gives 
\begin{align*}
(A_1\sqcap\dots\sqcap A_k)^I_J&=\frac{1}{k!}\sum_{\rho\in S_{k-1}}\epsilon(\rho)\sum_{\substack{J'\in P^*_{k-1}(J)\\I'\in P^*_{k-1}(I)}}
\left((-1)^{I'+J'}\det(\mathbf{A}_\rho)(A_k)^{I\setminus I'}_{J\setminus J'}\right).
\end{align*}
Note that the subsets $I\setminus I'$ and $J\setminus J'$ have only one element, so $(A_k)^{I\setminus I'}_{J\setminus J'}$ is an entry of the matrix $A_k$. 
By formula~\eqref{expansion} and the properties of the determinant of a matrix it follows that the right-hand side is equal to 
\[
\frac{1}{k!}\sum_{\rho \in S_{k-1}}\epsilon(\rho)\sum_{J'\in P^*_{k-1}(J)}\epsilon(\rho_{J'})\det(A_1(I,j_{\rho(1)})\mid\dots\mid A_{k-1}(I,j_{\rho(k-1)})\mid A_k(I,J\setminus J')),
\]
where $\rho_{J'}\in S_k$ is the permutation such that $j_{\rho_{J'}(k)}=J\setminus J'$ and fixes all other elements. Then the thesis follows since every permutation on $k$ elements is the product of a transposition taking the last element in a given position and a permutation on the others $k-1$ elements. 
\endproof

\begin{cor}\label{minors}
For $A:V\to V$ and $1\le k\le g$ let 
\begin{equation}\label{A^k}
A^{[k]}:=\underbrace{A\sqcap\dots \sqcap A}_{k\, times},
\end{equation}
Then we have $A^{[k]}=\bigwedge^k A$, where $\bigwedge^kA$ is defined as in~\eqref{wedge}.
\end{cor}



For any  
$A:\wedge^pV\to \wedge^pV$ and $B:\wedge^qV\to \wedge^qV$ we 
define the linear map 
\[A\ast B:\wedge^{g-(p+q)}V\to \wedge^{g-(p+q)}V\]
given by the matrix 
\begin{equation*}
(A\ast B)^I_J=(-1)^{I+J}(A\sqcap B)^{I^c}_{J^c},
\end{equation*}
for $I,\,J\in P^*_{g-(p+q)}(X_g)$.
If $A_1,\dots,A_k:V\to V$ are linear maps then the matrix of the map $A_1\ast \dots\ast A_k$, which we denote with the same symbol, has entries
\begin{equation}\label{def_ast}
(A_1\ast \dots\ast A_k)^I_J=(-1)^{I+J}(A_1\sqcap\dots\sqcap A_k)^{I^c}_{J^c},
\end{equation}
for $I,J\in P^*_{g-k}(X_g)$. 

For $A:V\to V$ and $1< k\le g$, 
\[(\underbrace{A\ast\dots \ast A}_{k\, times})^I_J=
(-1)^{I+J}\left(A^{[k]}\right)^{I^c}_{J^c}=\left(A^{(g-k)}\right)^I_J,\]
for $I,J\in P^\ast_{g-k}(X_g)$.
For example
\[A^{(1)}=\underbrace{A\ast\dots \ast A}_{g-1\, times}.\]

\begin{lem}\label{link}
If $v_1,\dots,v_k\in V$, then
\[v_1{}^tv_1\ast\dots\ast v_k{}^tv_k=\frac{1}{k!}\,(v_1\wedge\dots\wedge v_k)\,{}^t(v_1\wedge\dots\wedge v_k).\]
\end{lem}
\proof
For any $1\le k< g$, the Hodge $\ast$-operator gives an isomorphism 
\[\ast_H:\textstyle{\bigwedge}^kV\to\textstyle{\bigwedge}^{g-k}V.\]
If $e_I$ is the basis in~\eqref{basis_wedge} then the Hodge $\ast$-operator is defined by
\[\ast_H(e_I)=\epsilon(I,I^c)\,e_{I^c},\;I\in P^*_k(X_g),\] where $\epsilon(I,I^c)$ is the sign of the permutation that turns the set $I\cup I^c$ into the set $X_g$.

Define $A$ as the matrix whose $i$-th row is the vector $v_i$:  
\[A=\begin{pmatrix}
(v_1)_1&\dots&(v_1)_g\\
\vdots & & \vdots\\
(v_k)_1&\dots&(v_k)_g
\end{pmatrix}.\]
With respect to the basis $\{\ast_H(e_I)\}_{I\in P^*_k(X_g)}$ the coordinates of the vector $v_1\wedge\dots\wedge v_k$ are the following
\[(v_1\wedge\dots\wedge v_k)_J=\epsilon(J,J^c)\det(A_{J^c}), \quad J\in P^*_{g-k}(X_g),\]
where $A_{J^c}$ is the matrix obtained by $A$ by taking columns in $J^c$.

Let $V_i=v_i{}^tv_i$. A simple computation shows that $\epsilon(I,I^c)\epsilon(J,J^c)=(-1)^{I+J}$, hence by Lemma~\ref{lemma} it is enough to prove that for $I,\,J\in P^*_k(X_g)$ 
\begin{equation}\label{identity}
\sum_{\sigma\in S_k}\epsilon(\sigma)\det(\mathbf{V}_\sigma)=\det(A_I)\,\det(A_J),
\end{equation}
where 
\[\mathbf{V}_\sigma=(V_1(I,j_{\sigma(1)})\mid\dots\mid V_k(I,j_{\sigma(k)}))\] and $A^i_j=(v_i)_j$ as before.
Identity~\eqref{identity} easily follows by the fact that 
\[V_h(I,j_{\sigma(h)})=(v_h)_{j_{\sigma(h)}}\,{}^t(A(h,I)).\]
\endproof

\section{Singular modular forms and differential operators}\label{singular}

In this section we will present some properties of singular modular forms. We need to introduce the notion of multiplier system, since we will consider not only modular forms of integral weight but also half-integral weight ones. 

Let
\[\Gamma_g(n)=\Set{\gamma\in\Gamma_g\mid \gamma\equiv \mathbf{1}_{2g}\pmod n}.\]
A group $\Gamma\subset\Gamma_g$ is called a congruence subgroup if $\Gamma_g(n)\subset\Gamma$ for some $n$. 
A multiplier system of weight $r\in\RR$ for a congruence subgroup $\Gamma$ is a function $v:\Gamma\to\CC^\ast:=\CC\setminus\{0\}$ such that $j_r(\gamma,\tau):=v(\gamma)\det(C\tau+D)^{r}$ is holomorphic in $\tau$ and satisfies the following conditions:
\begin{enumerate}
\item[(i)]
$j_r(\gamma_1\gamma_2,\tau)=j_r(\gamma_1,\gamma_2\cdot\tau)j_r(\gamma_2,\tau)$ for all $\gamma_1,\,\gamma_2\in\Gamma$ and $\tau\in\HH_g$;
\item[(ii)]
$j_r(-\mathbf{1}_{2g},\tau)=1$, if $\mathbf{1}_{2g}\in\Gamma$.
\end{enumerate}

Let $V$ be a finite dimensional complex vector space and $\rho:\GL\to\op{GL}(V)$ be a rational representation of $\GL$. The irreducible ones are uniquely identified by their highest weight. We will write $\rho=(\lm_1,\dots,\lm_g)$ if $(\lm_1,\dots,\lm_g)$, with $\lambda_i\in\ZZ$ and $\lm_1\ge\dots\ge\lm_g$, is the highest weight of $\rho$. The dual representation of $\rho$ is $\rho^\vee:\GL\to\op{GL}(V^\vee)$ with $\rho^\vee(A)={}^t\rho(A\inv)$. For example the representation $\det=(1,\dots,1)$ and its dual representation is $\det^{-1}=(-1,\dots,-1)$. Then for any $k\in\ZZ$ the representation $\rho\otimes\det^k=(\lm_1+k,\dots,\lm_g+k)$. If $\rho=(\lm_1,\dots,\lm_g)$, the $\op{co\,-\,rank}$ of $\rho$ is defined as 
\[\op{co\,-\,rank}(\rho)=\#\{i\mid 1\le i\le g\; \text{with}\;\lambda_i=\lambda_g\}.\] 
The weight $w(\rho)$ of a representation is defined as the biggest integer $k$ such that $\det^{-k}\otimes\rho$ is a polynomial representation. If $\rho=(\lm_1,\dots,\lm_g)$ then $w(\rho)=\lm_g$. In a few words, the $\op{co\,-\,rank}$ of a representation counts how many times the weight appears as an entry of the highest weight. 

An irreducible representation is called reduced if its weight is 0. 
For $r\in\ZZ$ we will consider irreducible representations $\rho$ of the form
\begin{equation*}
\rho=\op{det}^{r/2}\otimes\rho_0,
\end{equation*}
where $\rho_0$ is a reduced irreducible representation. Such a representation is called a half-integral weight representation and it is called singular if $2w(\rho)<g$. 

If $v$ is a multiplier system of weight $r/2$ for a subgroup $\Gamma$ and $\rho=\op{det}^{r/2}\otimes\rho_0$ as before, a holomorphic function $f:\HH_g\to V$ is a vector-valued modular form with respect to $\Gamma$, $\rho$ and $v$ if
\[f(\gamma\cdot\tau)=v(\gamma)\rho(C\tau+D)\,f(\tau),\;\forall\gamma\in\Gamma,\,\forall\tau\in\HH_g.\]
If $g=1$ we need to require also that $f$ is holomorphic at $\infty$. 
Denote by $[\Gamma,\rho,v]$ the complex vector space of such modular forms. If $v$ is trivial it will be omitted in the notation. Each $[\Gamma,\rho,v]$ is a finite dimensional complex vector space. If the representation $\rho$ is decomposable, that is $\rho=\rho_1\oplus\rho_2$, then $[\Gamma,\rho,v]=[\Gamma,\rho_1,v]\oplus [\Gamma,\rho_2,v]$. So one usually restricts to the study of vector-valued modular forms with respect to irreducible representations.
If $\rho=\det^{r/2}$ we will denote by $[\Gamma,r/2,v]$ the vector space of modular forms with respect to $\rho$ and refer to its elements as scalar-valued modular forms of weight $r/2$ with a multiplier system $v$.

It is classically known that every modular form admits a Fourier expansion
\begin{equation}\label{Fourier}
f(\tau)=\sum_{S}a(S)e^{\pi i\op{Tr}(S\tau)},
\end{equation}
where $S$ runs trough a rational lattice of symmetric matrices and $a(S)\in V$. With these notations, a modular form is called singular if the matrices that appear in its Fourier expansion are singular matrices, that is $a(S)\neq0$ implies that $\det S=0$. 

The rank of a modular form is defined as
\[\op{rank}(f)=\op{max}\{\op{rank}(S)\,|\,a(S)\neq0\}.\]
Clearly $0\le \op{rank}(f)\le g$.

\begin{prop}[cf. \cite{F91,F93}]\label{lem_rank}
A non zero modular form $f\in[\Gamma,\rho]$ is singular if and only if $\rho$ is a singular representation. If $f$ is a singular modular form, then
\[\op{rank}(f)=2w(\rho).\]
If $h$ is a singular scalar-valued modular form, then
\[2w(h)\in \NN.\]
\end{prop}

%
%
%

By Proposition~\ref{lem_rank}, a scalar-valued modular form $f\in[\Gamma,r/2]$ is singular if and only if $r<g$. Moreover, if this is the case then $\op{rank}(f)=r$. We can characterize these properties by means of suitable differential operators. Let $\partial_{\tau_{ij}}:=\frac{\partial}{\partial\tau_{ij}}$ and define the $g\times g$ matrix of differential operators
\begin{equation}\label{diff operator}
\partial:=(\partial_{ij}), \qquad
\partial_{ij}=\frac{1+\delta_{ij}}{2}\,\partial_{\tau_{ij}},
\end{equation}
where $\delta_{ij}$ is the Kronecker delta.
For $1\le k \le g$ define the differential operator acting on $f$ as 
\[\partial^{[k]}f=(\det(\partial(I,J))\,f)_{I,J\in P^*_k(X_g)}.\] If $f$ has the Fourier expansion \eqref{Fourier} with $a(T)\in \CC$, then 
\[
\partial^{[k]}f=\sum a(S)\,S^{[k]}e^{\pi i \op{Tr}(S\tau)},
\] 
where $S^{[k]}$ is defined in \eqref{A^k}. Then  it follows by definition that $f$ is singular if and only if $\p g f=0$. Moreover 
\begin{equation}\label{rank}
\op{rank}(f)=n\Leftrightarrow \p gf=\p{g-1}f=\dots=\p{n+1}f=0\;\text{and}\; \p nf\neq0.
\end{equation}

\section{A new construction of vector-valued modular forms}\label{heart}

In this section we will work with scalar-valued modular forms with trivial multiplier system in order to ease notations. Nevertheless the same arguments work for scalar-valued modular forms with some non-trivial multiplier system with few changes. We will see an example of this in Section~\ref{sec:theta}.

For $f,\,h\in[\Gamma,k/2]$ let
\[A_{f,\,h}=f^2\,\partial\left(\frac{h}{f}\right)=f(\partial h)-(\partial f)h,\]
where $\partial$ is defined as in~\eqref{diff operator}. 
Then $A_{f,\,h}$ is a vector-valued modular form with respect to the group $\Gamma$ and the representation $\det^k\otimes\op{Sym^2}(\CC^g)=(k+2,k,\dots,k)$. More explicitly for any $\tau\in\HH_g$ and any $\gamma\in\Gamma$ it holds that
\[A_{f,\,h}(\gamma\cdot\tau)=\det(C\tau+D)^{k}\,(C\tau+D)\,A_{f,\,h}(\tau)\,{}^t(C\tau+D).\]
We will be interested in suitable products of this kind of vector-valued modular forms when $f$ and $h$ are weight $1/2$ scalar-valued modular forms. 
If we let 
\[\rho_k=(k+2,\dots,k+2,k,\dots,k)\]
 with $\op{co\,-\,rank}(\rho_k)=g-k$, then $A_{f,\,h}\in[\Gamma,\rho_1]$ if $f,\,h\in[\Gamma,1/2]$.

\begin{prop}\label{lemma vec-val}
If $A_1,\dots,A_k\in[\Gamma,\rho_1]$ then 
\[A_1\ast\cdots\ast A_k\in[\Gamma,\rho_k],\]
 where $\ast$ is defined as in~\eqref{def_ast}.
\end{prop}

\proof
By definition
\[(A_1\ast\cdots\ast A_k)(\gamma\cdot\tau)=\big(\rho_1(C\tau+D)A_1(\tau)\big)
\ast\cdots\ast\big(\rho_1(C\tau+D)A_k(\tau)\big).\]
So we need to prove that 
\[\big(\rho_1(C\tau+D)A_1(\tau)\big)
\ast\cdots\ast\big(\rho_1(C\tau+D)A_k(\tau)\big)=\rho_k(C\tau+D)(A_1\ast\cdots\ast A_k)(\tau).\]
It is enough to check the transformation rule for vector-valued modular forms of a given type.  Let $v_i:\HH_g\to V$ be such that
\[v_i(\gamma\cdot\tau)=\op{det}(C\tau+D)^{1/2}\,(C\tau+D)\,v_i(\tau),\;\forall\gamma\in\Gamma,\]
then $v_i{}^tv_i\in[\Gamma,\rho_1]$. 
By~\cite{sm89} we have that 
\[(v_1\wedge\dots\wedge v_k)\,{}^t(v_1\wedge\dots\wedge v_k)\in[\Gamma,\rho_k].\] The thesis then follows by Lemma~\ref{link} since
\[v_1{}^tv_1\ast\cdots\ast v_k{}^tv_k=\frac{1}{k!}(v_1\wedge\dots\wedge v_k)\,{}^t(v_1\wedge\dots\wedge v_k).\]
\endproof

Then by Proposition~\ref{lemma vec-val} it easily follows that if $f_i,\,h_i\in[\Gamma,1/2]$, $i=1,\dots,k$, then 
\[A_{f_1,\,h_1}\ast\dots\ast A_{f_k,\,h_k}\]
is a vector-valued modular form with respect to the irreducible representation $\rho_k$.

%
%

We will show that these vector-valued modular forms are related to a generalization of the pairing defined in~\cite{F75}. 

For any $1\le k\le g$ and $f,\,h\in[\Gamma,k/2]$ we define the pairings
\begin{align*}
\{f,h\}_k&=\sum_{p=0}^k(-1)^{p}\,\partial^{[p]}f\sqcap \partial^{[k-p]}h.
\\  
[f,h]_k&=\sum_{p=0}^k(-1)^{p}\,\partial^{[p]}f\ast \partial^{[k-p]}h.
\end{align*}
Note that $\{f,h\}_1=A_{f,\,h}$ 
 and \[[f,h]_k=(\{f,h\}_k)^{(g-k)}.\]

If $f,\,h\in[\Gamma,(g-1)/2]$, the $\Gamma$-invariant holomorphic differential form $\omega_{f,h}$ described in~\cite{F75} is 
\begin{equation}\label{omega}
\omega_{f,h}=\{f,h\}_{g-1}\sqcap d\check{\tau}=\op{Tr}([f,h]_{g-1}\,d\check{\tau}),
\end{equation}
where $d\check{\tau}$ is the basis of $\Omega^{N-1}(\HH_g)$ given in~\eqref{basis}.

In what follows we will focus on modular forms of half integral weight which are products of weight 1/2 ones. 
\begin{lem}\label{lemma_partial}
If $f,f_1,\dots,f_l\in[\Gamma,1/2]$, then for $k\in\NN$
\[\p k (f_1\cdots f_l)=
\begin{cases}
0&\text{if}\;k>l\\
k!\displaystyle\sum_{I=\{i_1<\dots<i_k\}}\frac{ f_{ 1}\cdots f_{ l}}{ f_{ {i_1}}\cdots f_{ {i_k}}}
\,\partial f_{ {i_1}}\sqcap\dots\sqcap\partial f_{ {i_k}}& \text{if}\;1\le k\le l
\end{cases}\]
and
\[
\p k f^l=\begin{cases}
0&\text{if}\;k>l\\
l(l-1)\cdots(l-k+1)f^{l-k}(\partial f)^{[k]}& \text{if}\;1\le k\le l
\end{cases},\]

\noindent where $\partial$ is defined as in~\eqref{diff operator}.
\end{lem}
\proof
The product $\sqcap$ is bilinear, commutative, associative and distributive with respect to the sum of matrices (cf.~\cite{F75}). Hence for $A:\wedge^pV\to \wedge^pV,\;B:\wedge^qV\to \wedge^qV$, the following formula holds:
\[(A+B)^{[k]}=\sum_{j=0}^k \binom{k}{j}A^{[j]}\sqcap B^{[k-j]}.\]
From this, it easily follows that for $f\in[\Gamma,k]$ and $h\in[\Gamma,l]$ and for every $1\le p\le g$ 
\begin{equation}\label{der_prodotto}
\partial^{[p]}(fh)=\sum_{j=0}^p \binom{p}{j}\p j f\sqcap \p{p-j} h.
\end{equation}
Note that here the terms for which $j>\op{rank}(f)$ or $p-j>\op{rank}(h)$ vanish by \eqref{rank}.
So the thesis follows by Lemma \ref{lem_rank} and formula \eqref{der_prodotto}.
\endproof

\begin{prop}\label{lem_pairing}
Let $1\le k< g$. If $f_i,\,h_i\in[\Gamma,1/2]$, $i=1,\dots,k$ then
\[[f_1\cdots f_k,h_1\cdots h_k]_k=\sum_{\sigma\in S_k} A_{f_1,\,h_{\sigma(1)}}\ast\dots\ast A_{f_k,\,h_{\sigma(k)}},\]
where $S_k$ is the group of permutations on $k$ elements.
\end{prop}

\proof
It is enough to prove that
\[\{f_1\cdots f_k,h_1\cdots h_k\}_k=\sum_{\sigma\in S_k} A_{f_1,\,h_{\sigma(1)}}\sqcap\dots\sqcap A_{f_k,\,h_{\sigma(k)}}.\]
If $I=\{i_1,\dots,i_p\}\in P^*_p(X_k)$ and $\sigma\in S_p$ denote by $\sigma(I)=\{i_{\sigma(1)},\dots,i_{\sigma(p)}\}$. Moreover, denote by $f_I=f_{i_1}\cdots f_{i_p}$ and by $\partial f_I=\partial f_{i_1}\sqcap\dots\sqcap\partial f_{i_p}$. Then by Lemma~\ref{lemma_partial} we have
\[\p p(f_1\cdots f_{k})=p!\sum_{I\in P^*_p(X_k)}\!\!\!\!\!\!f_{I^c}\,\partial f_{I}.\]

Since the product $\sqcap$ is bilinear it holds that
\[A_{f_1,\,h_{1}}\sqcap\dots\sqcap A_{f_k,\,h_{k}}=\sum_{p=0}^k(-1)^p\sum_{I\in P^*_p(X_k)}
h_I\,f_{I^c}\,(\partial f_I\sqcap\partial h_{I^c}).\]
Then
\begin{align*}
\sum_{\sigma\in S_k} A_{f_1,\,h_{\sigma(1)}}\sqcap\dots \sqcap A_{f_k,\,h_{\sigma(k)}}&=\sum_{p=0}^k(-1)^p 
\!\!\!\!
\sum_{I\in P^*_p(X_k)}
\!\!\!\!
f_{I^c}\partial f_I\sqcap\left(
\sum_{\sigma\in S_k}\!\!h_{\sigma(I)}\,\partial h_{\sigma(I^c)}\right)=\\
&\!\!\!\!\!\!\!\!\!\!\!\!\!\!\!\!
= \sum_{p=0}^k(-1)^p 
\!\!\!\!
\sum_{I\in P^*_p(X_k)}
\!\!\!\!
f_{I^c}\partial f_I\sqcap\left(p!(k-p)
!\!\!\!\!
\sum_{J\in P^*_p(X_k)}
\!\!\!\!
h_J\,\partial h_{J^c}\right)=\\
&\!\!\!\!\!\!\!\!\!\!\!\!\!\!\!\!
=\sum_{p=0}^k(-1)^{p}\,\partial^{[p]}(f_1\cdots f_{k})\sqcap \partial^{[k-p]}(h_1\cdots h_{k}).
\end{align*}
\endproof

\begin{cor}\label{cor_pairing}
Let $1\le k< g$. If $f,\,h\in [\Gamma,1/2]$ then  
\[[f^k,h^k]_k=k!(A_{f,h})^{(g-k)}.\]
\end{cor}

As a consequence, if $f=F^{g-1}$ and $h=H^{g-1}$ for $F,H\in[\Gamma,1/2]$ it easily follows that
\[\omega_{f,\,h}=(g-1)!\op{Tr}((A_{F,\,H})^{(1)} d\check{\tau}),\]
where $\omega_{f,\,h}$ is defined in~\eqref{omega}.
So we recover the result in~\cite[Theorem 14]{FDP} and actually generalize it to every $\Gamma$-invariant holomorphic differential form~\eqref{omega} constructed from two singular scalar-valued modular forms of weight $(g-1)/2$ which are products of weight $1/2$ ones.

\begin{rem}
For $k=g$ the identities in Proposition~\ref{lem_pairing} and Corollary~\ref{cor_pairing} still hold. The products $f_1\cdots f_g$ and $h_1\cdots h_g$ are no more singular modular forms and we are constructing scalar-valued modular forms of weight $g+2$ instead of vector-valued modular forms. In particular one of the scalar-valued modular forms we obtain is 
\[\det(A_{f,\,h})=g!\sum_{p=0}^g(-1)^{p}\,\partial^{[p]}
(f^g)\sqcap \partial^{[g-p]}(h^g).\]
\end{rem}

\section{Theta functions and theta constants}\label{sec:theta}

In the first part of this section we give a brief introduction to the theory of theta functions and theta constants. We will use second order theta constants to give examples of the vector-valued modular forms of the previous section constructed with scalar-valued modular forms with some non-trivial multiplier system.  

Finally we will give a construction of vector-valued modular forms with gradients of odd theta functions. We will prove that the two methods, although so different at the first look, give rise to elements of the same vector space of vector-valued modular forms. 

Let us start introducing theta functions. For a vector $m=\ch {m'}{m''}$, $m',\,m''\in\ZZ^g$,  the theta function with characteristic $m$ is defined by the series 
\begin{equation*}
\vartheta_m(\tau,z)=
\sum_{n\in\ZZ^g}\mathtt{exp}\left(\frac{1}{2} {}^t(n+m'/2)\tau(n+m'/2)+{}^t(n+m'/2)(z+m''/2)\right),
\end{equation*}
where $\mathtt{exp}(\cdot)=e^{2\pi i(\cdot)}$. This series converges absolutely and uniformly in every compact subset of $\HH_g\times\CC^g$. Then it defines a holomorphic function $\vartheta_m:\HH_g\times\CC^g\to\CC$. It is an even or odd function of $z$ if ${}^tm'm''$ is even or odd respectively. Correspondingly, the characteristic $m$ is called even or odd. 

Since we have 
\[\vartheta_{m+2n}(\tau,z)=(-1)^{{}^tm'\,n''}\vartheta_m(\tau,z),\]
for any $n\in\ZZ^g$, we can normalize a characteristic by the condition that its coefficients are either zero or one. 

The integral symplectic group acts not only on the Siegel upper-half space by~\eqref{symplectic action} but also on the set of characteristics by the following formula. For $\gamma\in\Gamma_g$ and a characteristic $m\in\Set{0,1}^{2g}$ set
\begin{equation*}
\gamma\cdot\ch{m'}{m''}=\left[\sm{D&-C}{-B&A}\sm{m'}{m''}+\sm{\diag(C^tD)}{\diag(A^tB)}\right] \pmod{2},
\end{equation*}
where we think of the elements of $\ZZ/2\ZZ$ as zeroes and ones. 
The action defined in this way is neither linear nor transitive. Indeed, the action preserves the parity of the characteristics. Clearly the action of the principal congruence subgroup $\Gamma_g(2)$ on the set of theta characteristics is trivial.

Theta functions with characteristics satisfy the following transformation law for any $\gamma\in\Gamma_g$ (see~\cite{ig1}):
\begin{equation*} 
\vartheta_{\gamma\cdot m}(\gamma\cdot\tau,{}^t(C\tau+D)^{-1}z)=
\kappa(\gamma)\texttt{exp}\left(\phi_m(\gamma)+\dfrac{1}{2}{}^tz(C\tau+D)^{-1}Cz\right)
\det(C\tau+D)^{1/2}\vartheta_m(\tau,z),
\end{equation*}
where $\kappa(\gamma)$ is an 8\textsuperscript{th} root of unity with the same sign ambiguity as $\op{det}(C\tau+D)^{1/2}$ and 
\begin{equation*}
\phi_m(\gamma)  =  -\frac{1}{8} ({}^t m ' \, {}^tB D m' + {}^t m '' \, {}^t A C m'' - 2{}^t m ' \, {}^t B C m'') + \frac{1}{4} {}^t\op{diag}(A_{}^tB)(D m ' - C m '').
\end{equation*}
The theta constant with characteristic $m\in\Set{0,1}^{2g}$ is defined as
\[\vartheta_m(\tau)=\vartheta_m(\tau,0).\]
A theta constant does not vanish identically if and only if the characteristic is even. 
These will give scalar-valued modular forms with respect to the congruence subgroups
\[\Gamma_g(n,2n)=\Set{\gamma\in\Gamma_g(n)\;|\;
\op{diag}(C)\equiv\op{diag}(B)\equiv0\pmod{2n}},\]
for $n=2,4$. We recall that $\Gamma_g(n)\subset\Gamma_g$ is the kernel of the reduction modulo $n$.

For every $\gamma\in\Gamma_g(2)$ a theta constant with even characteristic $m$ satisfies the following transformation formula
\begin{equation}\label{trans_theta}
\vartheta_m(\gamma\cdot\tau)=\kappa(\gamma)\texttt{exp}(\phi_m(\gamma))\op{det}(C\tau+D)^{1/2}\vartheta_m(\tau).
\end{equation}
By a direct computation it follows that $\texttt{exp}(\phi_m(\gamma))=1$ for every $\gamma\in\Gamma_g(4,8)$ hence $\vartheta_m(\tau)$ is a scalar-valued modular form of weight 1/2 with respect to the congruence subgroup $\Gamma_g(4,8)$ and the multiplier system $\kappa(\gamma)$. It is well known that $\kappa(\gamma)^4=1$ for any $\gamma\in\Gamma_g(2,4)$. Define $\Gamma_g(2,4)^\ast$ as the index two subgroup of $\Gamma_g(2,4)$ where $\kappa(\gamma)^2=1$.

As weight $1/2$ scalar-valued modular forms, theta constants have rank 1 (see Section~\ref{singular}). But in this particular case the fact that  $\op{rank}(\vartheta_m)=1$ for even $m$ is a straightforward consequence of  the following system of equations, usually called ``heat equation'':
\begin{equation}\label{heat}
\frac{\partial^2}{\partial z_j\partial z_k}\vartheta_m(\tau,z)=2 \pi i (1+\delta_{jk})\,\frac{\partial}{\partial\tau_{jk}}\vartheta_m(\tau,z),\;j,k=1,\dots,g.
\end{equation}

\subsection{Vector-valued modular forms from second order theta constants}\label{sec:second_order}

For any $\e\in\Set{0,1}^g$ the second order theta functions are defined as
\[\T[\e](\tau,z)=\vartheta\ch\e0(2\tau,2z).\]
These are all even functions of $z$ and are related to theta functions with characteristic by Riemann's addition formula
\[\T[\al](\tau,z)\,\T[\al+\e](\tau,0)=\frac{1}{2^g}
\sum_{\sigma\in\Set{0,1}^g}(-1)^{\al\cdot\sigma}\vartheta\ch\e\sigma(\tau,z)^2,\]
for any $\al,\,\e\in\Set{0,1}^g$.

As for theta constants with characteristic, for $\e\in\Set{0,1}^g$ denote by $\T[\e]=\T[\e](\tau,0)$ the second order theta constant with characteristic $\e$. 
These are related to theta constants with characteristic by the following equations (cf.~\cite{ig1}):
\begin{align}\label{R_addition}
\Theta[\sigma](\tau)\,\Theta[\sigma+\e](\tau)&=\frac{1}{2^g}\sum_{\de\in\Set{0,1}^g}(-1)^{\sigma\cdot\de}\tch{\e\\\de}(\tau)^2,\\[6pt]
\label{R_addition1}
\tch{\e\\\de}(\tau)^2
&=\sum_{\sigma\in\Set{0,1}^g}(-1)^{\sigma\cdot\de}
\Theta[\sigma](\tau)\,\Theta[\sigma+\e](\tau).
\end{align}
  
For every $\gamma\in\Gamma_g$ let $\tilde{\gamma}\in\Gamma_g$ be such that $2(\gamma\cdot\tau)=\tilde{\gamma}\cdot(2\tau)$, that is $\tilde{\gamma}=\sm{A&2B}{C/2&D}$ if $\gamma=\sm{A&B}{C&D}$.
By the above transformation formula for theta constants we get
\begin{equation*}
\Theta[\e](\gamma\cdot\tau)=\kappa(\tilde{\gamma})\det(C\tau+D)^{1/2}\Theta[\e](\tau),\;\forall\gamma\in\Gamma_g(2,4).
\end{equation*}
Second order theta constants are then modular forms of weight $1/2$ with respect to the congruence subgroup $\Gamma_g(2,4)$ and the multiplier system $k(\tilde{\gamma})$. 
By equations~\eqref{R_addition} and~\eqref{R_addition1} it is easy to see that $\kappa(\tilde{\gamma})^2=\kappa(\gamma)^2$.

For every $\e,\,\de\in\Set{0,1}^g$ denote by $A_{\e\,\de}:=\{\T[\e],\T[\de]\}_1$. 
Then it is easy to see that 
\[A_{\e\,\de}(\gamma\cdot\tau)=\kappa(\gamma)^2\,\det(C\tau+D)\,(C\tau+D)\,A_{\e\,\de}(\tau)\,{}^t(C\tau+D),\;\forall\gamma\in\Gamma_g(2,4).\]
By this equation and Proposition~\ref{lemma vec-val}, for $\e_1,\dots,\e_k,\de_1,\dots,\de_k\in\Set{0,1}^g$, the vector-valued modular form $A_{\e_1\,\de_1}\ast\dots\ast A_{\e_k\,\de_k}$ satisfies the  following transformation formula for any $\gamma\in\Gamma_g(2,4)$
\begin{equation}\label{Aedek}
(A_{\e_1\,\de_1}\ast\dots\ast A_{\e_k\,\de_k})(\gamma\cdot\tau)=
\kappa(\gamma)^{2k}\,\rho_k(C\tau+D)(A_{\e_1\,\de_1}\ast\dots\ast A_{\e_k\,\de_k})(\tau).
\end{equation}
Note that as we said before, here we are dealing with scalar-valued modular forms with some non-trivial multiplier system that shows up in the transformation formula~\eqref{Aedek} for the vector-valued modular form we are constructing.

Since $\kappa(\gamma)^4=1$ for any $\gamma\in\Gamma_g(2,4)$ and $\kappa(\gamma)^2=1$ for every $\gamma\in\Gamma_g(2,4)^\ast$, then
$A_{\e_1\,\de_1}\ast\dots\ast A_{\e_k\,\de_k}\in [\Gamma_g(2,4)^\ast,\rho_k]$. If $k$ is even, then $A_{\e_1\,\de_1}\ast\dots\ast A_{\e_k\,\de_k}\in [\Gamma_g(2,4),\rho_k]$.

\subsection{Vector-valued modular forms from gradients of odd theta functions}\label{sec:gradients}

For any odd characteristic $n$ denote the gradient as
\begin{equation}\label{graddefine}
v_n(\tau):=\grad_z\theta_n(\tau,z)|_{z=0},
\end{equation}
considered as a column vector. 
Then for every $\gamma\in\Gamma_g(2)$, $v_n(\tau)$ satisfies the following transformation formula:
\begin{equation*}
v_n(\gamma\cdot\tau)=\kappa(\gamma)\texttt{exp}(\varphi_n(\gamma))\det(C\tau+D)^{1/2}(C\tau+D)v_n(\tau),
\end{equation*}
where $\kappa(\gamma)$ and $\phi_n(\gamma)$ are exactly the same factors appearing in the transformation formula~\eqref{trans_theta} for theta constants. 

For a matrix $N=(n_1,\dots,n_k)\in M_{2g\times k}$ with $\{n_i\}_{i=1,\dots,k}$ a set of distinct odd characteristics define
\[W(N)(\tau)=\pi^{-2k}\,(v_{n_1}(\tau)\wedge\ldots\wedge v_{n_{k}}(\tau))\,{}^t(v_{n_1}(\tau)\wedge\ldots\wedge v_{n_{k}}(\tau)).\]
By~\cite{sm89} $W(N)$ is non-vanishing and satisfies the following transformation rule for every $\gamma\in\Gamma_g(2)$:
\[W(N)(\gamma\cdot\tau)=\kappa(\gamma)^{2k}\,\texttt{exp}\left(2\sum_{i=1}^k\phi_{n_i}(\gamma)\right)
\rho_k(C\tau+D)W(M)(\tau),\]
where $\rho_k=(k+2,\dots,k+2,k,\dots,k)$ as in Proposition~\ref{lemma vec-val} and $\kappa(\gamma)$ and $\phi_{n_i}(\gamma)$ are the same as in the transformation formula for theta constants~\eqref{trans_theta}.

If $M=(m_1,\dots,m_{2k})$ is a matrix of even characteristics, let 
\begin{equation*}
\vartheta_M=\vartheta_{m_1}\cdots\vartheta_{m_{2k}}.
\end{equation*}
For $\widetilde M=(M,M)$, then for any $\gamma\in\Gamma_g(2)$
\[\vartheta_{\widetilde M}(\gamma\cdot\tau)=\kappa(\gamma)^{2k}\,\texttt{exp}\left(2\sum_{i=1}^k\phi_{m_i}(\gamma)\right)\det(C\tau+D)^{2k}\vartheta_{\widetilde M}(\tau).\]

Hence in order to study the modularity of $W(N)$ we can 
look at modularity conditions for suitable products of theta constants. 

By~\cite{igusa} we have that
\begin{align*}
\vartheta_M\in[\Gamma_g(2,4),k]&\;\;\text{if}\;\; M{}^tM\equiv k\,\sm{0&\mathbf{1}_g}{\mathbf{1}_g&0}
\pmod 2,\\
\vartheta_M\in[\Gamma_g(2,4)^\ast,k]&\;\;\text{if}\;\; M{}^tM\equiv0\pmod 2\;\,\text{or}\,\; M\,{}^tM\equiv\sm{0&\mathbf{1}_g}{\mathbf{1}_g&0}\pmod 2.
\end{align*}
Since 
\[\widetilde N\,{}^t\widetilde N=2\sum_{i=1}^k n_i\,{}^t n_i\equiv 0\pmod 2,\]
where $\widetilde N=(N,N)$,
we conclude that $W(N)\in[\Gamma_g(2,4)^\ast,\rho_k]$ for any $1\le k\le g$. If $k$ is even then $W(N)\in[\Gamma_g(2,4),\rho_k]$.

\subsection{An identity of vector spaces of vector-valued modular forms}
In this section we will prove that the vector-valued modular forms introduced in Section~\ref{sec:second_order} and Section~\ref{sec:gradients} belong to the same vector space. A fundamental step in the proof is the following proposition that shows a consequence of the classical Riemann's addition theorem for theta functions.

\begin{prop}[\cite{GSM1,GSM}]\label{GSM_prop}
For an odd characteristic $n=\ch \ep\de$ 
\begin{equation}\label{CintermsofA}
v_{n}{}^tv_{n}=\frac{1}{4} \sum_{\al\in\Set{0,1}^g}(-1)^{\al\cdot\de}\,A_{ \ep+\al \,\al},
\end{equation}
where $v_n$ is the gradient of an odd theta function defined in~\eqref{graddefine}. 

Moreover for given $\e,\,\de\in\Set{0,1}^g$, denote by $n_\al=\ch{\ep+\de}{\al}$ for $\al\in\Set{0,1}^g$. Then 
\begin{equation}\label{AintermsofC}
A_{\ep\,\de}=\frac{1}{2^{g-2}}\sum_{\substack{\al\in\Set{0,1}^g\,\text{s.t.}\\\,n_\al\op{odd}}}(-1)^{\de\cdot\al}\;
v_{n_\al}{}^tv_{n_\al}.
\end{equation}
\end{prop}

This proposition fits in the big subject of generalizations of Jacobi's derivative formula. For $g=1$ the classical Jacobi identity states that
\[D\left(\ch{1}{1}\right)=-\,\tch{0\\0}\,\tch{1\\0}\,\tch{0\\1}.\]

Essentially, the problem of generalizing this formula consist in expressing the Jacobian determinant of $g$ distinct odd theta functions as a polynomial in theta constants.

In~\cite{Igu81} it is proven that the Jacobian determinant is always a rational function of the theta constants. The question about the possible expression as a polynomial in theta constants is more complicated. For $g=2$ the formula is still classical and gives the following. If $n_1,\dots,n_6$ are the six odd characteristics and $m_i=n_1+n_2+n_{i+2}$ for $i=1,\dots,4$ then
\[D(n_1,n_2)=\pm\,\vartheta_{m_1}\cdots\vartheta_{m_4}.\]
For $g=3$ it is known that if $n_i$, $i=1,\dots,3$ are odd characteristics the Jacobian determinant $D(n_1,n_2,n_3)$ is a polynomial in the theta constants if and only if $n_1+n_2+n_3$ is an even characteristic.
In higher degree there is a conjectural formula which has been proven only for $g\le 5$.

Nevertheless we can say when a Jacobian determinant is not a polynomial in theta constants  by looking at a condition on the characteristics involved. A set of characteristics is called essentially independent if the sum of any even number of them is not congruent to $\displaystyle{0}$ mod $2$. A triplet of odd characteristics is called azygetic or syzygetic if their sum is even or odd respectively. A set of odd characteristics is azygetic or syzygetic if all triples in the set are azygetic or syzygetic respectively.
By~\cite{Igu81} we know that if $n_1,\dots,n_g$ is a set of odd characteristics which is an essentially independent syzygetic set then $D(n_1,\dots,n_g)$ is not a polynomial in the theta constants.

A different generalization can be done by looking at higher order derivatives of theta functions. This is the direction taken in Proposition~\ref{GSM_prop} (recall that by the heat equation~\eqref{heat} one has that $4 \pi i \partial_{jk}=\partial_{z_j}\partial_{z_k}$).

Now we can establish our result about the identity of vector spaces of vector-valued modular forms. 
\begin{thm}\label{same_space1}
Denote by $V_{grad}$ the vector space of vector-valued modular forms generated by the modular forms $W(N)$, where $N$ is a matrix of $k$ distinct odd characteristics.
Denote also by $V_\Theta$ the vector space of vector-valued modular forms generated by the modular forms $A_{\e_1\,\de_1}\ast\dots\ast A_{\e_k\,\de_k}$ where $\e_1,\dots,\e_k,\de_1,\dots,\de_k\in\Set{0,1}^g$. 
Then for any $1\le k<g$ one has the identity of vector spaces
\[V_\T=V_{grad}.\]
\end{thm}

\begin{proof}
%
We will prove that each vector-valued modular form $A_{\e_1\,\de_1}\ast\dots\ast A_{\e_k\,\de_k}$ for some $\e_i,\,\de_i\in\Set{0,1}^g$, $i=1,\dots,k$ is in $V_{grad}$. By formula~\eqref{AintermsofC} we have that
\[A_{\ep_i\,\de_i}=\frac{1}{2^{g-2}}\sum_{\substack{\al\in\Set{0,1}^g\,\text{s.t.}\\\,n_{\al}^i\op{odd}}}(-1)^{\de\cdot\al}\;
v_{n_{\al}^i}{}^tv_{n_{\al}^i},\]
where $n_{\al}^i=\ch{\e_i+\de_i}{\al}$, for $i=1,\dots,k$.
By the linearity of the product $\ast$ and by applying Lemma~\ref{link} we see that 
there exists a computable constant $c$ such that
\[A_{\e_1\,\de_1}\ast\dots\ast A_{\e_k\,\de_k}= c
\sum_{\substack{\al_i\in\Set{0,1}^g\,\text{s.t.}\\\,[\ep_i+\de_i,\al_i]\op{odd}}}(-1)^{\sum_i \de_i\al_i}\;
W([\ep_1+\de_1,\,\alpha_{1}],\dots ,[\ep_k+\de_k,\,\alpha_{k}]).\]
On the other hand, by Lemma~\ref{link} and equation~\eqref{CintermsofA} it is also easy to prove that each vector-valued modular form $W(N)$ is in $V_\T$. This completes the proof. 
\end{proof}

\end{document}